\documentclass[leqno]{article}
\usepackage{amsmath}
\usepackage[latin1]{inputenc}
\usepackage{amssymb}
\usepackage{xcolor}
\usepackage[backref=page]{hyperref}
\usepackage{amsthm} 
\usepackage{enumerate}                                        
\usepackage{url}
\usepackage{graphicx}
\usepackage{epsfig}
\newtheorem{theorem}{\bf Theorem}

\newtheorem{corollary}[theorem]{\bf Corollary}
\newtheorem{lemma}[theorem]{\bf Lemma}

\newtheorem{proposition}[theorem]{\bf Proposition}
\newtheorem{definition}[theorem]{\bf Definition}

\numberwithin{equation}{section}
\numberwithin{theorem}{section}
\numberwithin{figure}{section}

\def\R{\mathbb{R}}
\def\L{\mathbb{L}}
\def\R{\mathbb{R}}

\begin{document}
\renewcommand{\thefootnote}{}
\footnotetext{Research partially supported by Ministerio de Econom\'ia y Competitividad Grant No: MTM2016-80313-P and the `Maria de Maeztu'' Excellence
Unit IMAG, reference CEX2020-001105-M, funded by
MCIN/AEI/10.13039/501100011033} 

\title{Calabi-Bernstein type results for critical points of a weighted  area functional in $\R^3$ and $\L^3$} 
\author{$\text{A. Mart\'inez}^{1}\text{ and A.L. Mart\'inez-Trivi\~no}^{2}$}
\vspace{.1in}
\date{}
\maketitle
{
\noindent $^1$Departamento de Geometr\'\i a y Topolog\'\i a, Universidad de Granada, E-18071 Granada, Spain\\ \\
e-mails: $\text{amartine@ugr.es}^{1}$ and $\text{aluismartinez@ugr.es}^{2}$}
\begin{abstract}
In this paper we prove some Calabi-Bernstein type  and non-existence results concerning complete $[\varphi,\vec{e}_3]$-minimal surfaces in $\R^3$ whose Gauss maps lie on compacts subsets of open hemispheres of $\mathbb{S}^2$. We also give a general non-existence result for complete spacelike $[\varphi,\vec{e}_3]$-maximal surfaces in $\L^3$ and, in particular, we obtain a Calabi-Bernstein type result when $\dot{\varphi}$ is bounded.
\end{abstract}
\vspace{0.2 cm}

\noindent 2020 {\it  Mathematics Subject Classification}: {53C42, 35J60  }

\noindent {\it Keywords: } Calabi-Bernstein, $[\varphi,\vec{e}_{3}]$-minimal surface, spacelike $[\varphi,\vec{e}_{3}]$-maximal surface, Bochner formula, minimal surfaces, maximum principles. 
\everymath={\displaystyle}
\section{Introduction.}
At the beginning of the 20th-century, Bernstein \cite{B} proved that the only entire minimal graph $\Sigma$ in $\mathbb{R}^{3}$ is the hyperplane. De Giordi \cite{G} extended this result to $\mathbb{R}^{4}$ and Almgren \cite{A} to $\mathbb{R}^{5}$. In fact, the Bernstein's Theorem is true until $\mathbb{R}^{7}$, see for instance \cite{Simons} and  Bombieri, De Giorgi and Giusti \cite{BGG} gave a counterexample for $\mathbb{R}^{n}$ with $n\geq 8$. However, under the additional assumption that the norm of the gradient of our graph is uniformly bounded, Moser \cite{Moser} proved that any entire minimal graph must be a hyperplane for any dimension.
\

In 1959, a generalized version of the Bernstein problem was given by Osserman  \cite{Oss},  who proved  the Nirenberg's conjecture, namely, he showed that any complete minimal surface  in $\mathbb{R}^{3}$ whose  Gauss map omits a neighborhood of some point at the sphere,  must be a plane. This result was generalized to $\mathbb{R}^{n}$ by Chern \cite{Chern}. Few years later, Fujimoto \cite{Fuj} showed that the plane is the only complete orientable minimal in $\mathbb{R}^{3}$ whose Gauss map omits at least 5 points of the sphere.
\

This kind of results have been also proved for other families of surfaces in different ambient spaces. For instance,  Hoffman, Osserman and Schoen \cite{HOS} proved a version of the Osserman's result for CMC surfaces. Barros, Aquino and Lima \cite{BAL} for complete CMC hypersurfaces in the hyperbolic space with prescribed Gauss map and Fern\'andez, G\'alvez and Mira \cite{FGM} for elliptic Weingarten complete multigraphs with a quasiconformal Gauss map that omits an open hemisphere.

In 1970 and for Lorentzian ambient spaces, Calabi \cite{CA} proved   that the only  entire maximal graph in the 3-dimensional Lorentz-Minkowski space $\mathbb{L}^{3}$ is a plane. Aledo, Rubio and Romero \cite{ARR} obtained that the only complete maximal surface in $\mathbb{L}^{3}$ is a spacelike plane.
 In 2003,  Al\'ias and Mira \cite{AM} gave a generalization to higher dimension and in 2009, Albujer and Al\'ias \cite{AA}  proved the Calabi-Bernstein result for maximal surfaces in Lorentzian product spaces.

\

The main goal of this work is to prove  Calabi-Bernstein's type results for  critical points of the  weighted area functional 
\begin{equation}
\label{area}
\mathcal{A}^{\varphi}(\Sigma)=\int_{\Sigma}e^{\varphi}\, d \Sigma
\end{equation}
on isometric immersions of  Riemannian  surfaces $\Sigma$ in a domain $\mathfrak{D}^3$ of $ \mathbb{R}^3$ (or $\mathbb{L}^3$) when $\varphi$ is the restriction on $\Sigma$ of a smooth function depending only on the  coordinate of $\mathfrak{D}^3$ in the direction of  $\vec{e}_3=(0,0,1)$ and where $d\Sigma$ denotes the volume element induced on $\Sigma$ by the  Euclidean (or Lorentzian) metric $\langle\cdot,\cdot\rangle:= dx^2+dy^2+dz^2$ ($\langle\cdot,\cdot\rangle_{\mathbb{L}^3}:= dx^2+dy^2-dz^2$). 

The Euler-Lagrange equation of \eqref{area} is given in terms of the mean curvature vector $\textbf{H}$ of $\Sigma$ as follows
\begin{equation}
\label{meancurvature}
\textbf{H}=(\overline{\nabla}\varphi)^{\perp} = \dot{\varphi} \ \vec{e}_3^{\,\perp},
\end{equation}
here $\perp$ denotes the projection on the normal bundle, and $\overline{\nabla}$ stands the usual gradient operator in $\mathbb{R}^{3}$ (or $\mathbb{L}^3$).
\\

Any critical point of \eqref{area}  in $\mathbb{R}^{3}$ (or $\mathbb{L}^3$)   will be called  {\sl $[\varphi, \vec{e}_3]$-minimal } (or {\sl spacelike $[\varphi, \vec{e}_3]$-maximal}) surface.
Interesting examples of these families of surfaces are:
\begin{itemize}
\item the classical minimal surfaces in $\mathbb{R}^{3}$ and the the spacelike maximal surfaces in $\mathbb{L}^{3}$ when $\varphi$ is a constant.
\item the translating solitons:  if $\varphi$ is just the height function, $\varphi(p)=\langle p,\vec{e}_3\rangle$, that is,  surfaces such that $$ t \rightarrow \Sigma + t \vec{e}_3 $$
is a mean curvature flow, i.e. the normal component of the  velocity at  each point is equal to the mean curvature at that point. 
\item the singular $\alpha$-minimal (spacelike $\alpha$-maximal) surfaces: if  $\varphi(p) =\alpha \log \langle p,\vec{e}_3\rangle$, $\alpha$=const.  For surfaces in  $\mathbb{R}^3$ and when  $\alpha=1$,  $\Sigma$ describes the shape of a ``hanging roof'', i.e. a heavy surface in a gravitational field that, according to the architect F. Otto \cite[p. 290]{Otto} are of importance for the construction of perfect domes. 
\end{itemize} 
This kind of surfaces has been widely studied specially from the viewpoint of calculus of variations. Classical results about the Euler equation and the existence and regularity for  the solutions of the Plateau problem for \eqref{area} can be found in \cite{BHT, H1, H2, HK, T}.
But the situation have changed in very recent years and nowdays  very and interesting geometric properties of these  surfaces are known. We mention, the convexity result of Spruck-Xiao \cite{SX} for complete mean-convex translating solitons which solved a Wang conjeture \cite{Wa},  they proved that the only entire vertical graph translating soliton is the rotationally symmetric Bowl soliton.  Hoffman, Ilmanen, Mart\'in y White \cite{HIMW} have classified all complete vertical graphs translating soliton in $\mathbb{R}^{3}$. 

B\"ome, Hildebrant and Tausch \cite{BHT} have characterized the 2-dimensional singular $\alpha$-catenary and Dierkes \cite{D} together with L\'opez \cite{RL} have classified the singular $\alpha$-catenaries and the rotationally symmetric singular $\alpha$-minimal examples.
\

Concerning  to Bernstein's problem, Bao-Shi \cite{BS}  proved that if $\Sigma$ is a convex complete translating soliton in $\mathbb{R}^3$ whose Gauss map lies in a closed ball of $\mathbb{S}^{2}$ of radii less than $\pi/2$, then $\Sigma$ must be a vertical plane, Kunikawa \cite{K}  generalized this result for arbitrary codimension  and Qiu \cite{HQ} have also given  a Bernstein type result for vertical translating soliton graphs in $\mathbb{R}^{3}$.
\\

Having  good  control over the geometry of a  $[\varphi,\vec{e}_{3}]$-minimal surface in $\mathbb{R}^3$ make it necessary to impose some analytical constraints on the function $\varphi$, see for instance \cite{MM,MM1,MM2,MMJ,MJ}. As a general condition we are going to consider  $\varphi:]a,b[\rightarrow \mathbb{R}$,  satisfying either 
\begin{align}
& \text{$\varphi$ monotone,}\quad \dddot{\varphi\hspace{0pt}} \dot{\varphi}\geq 0 \ \ \text{and} \ \  \ddot{\varphi}\leq 0 \quad \text{on $]a,b[$,}\label{hy1}
\end{align}
or
\begin{align}
&\text{$\varphi$ monotone,}\quad \text{and} \quad \ddot{\varphi}\geq 0 \quad \text{on $]a,b[$}.\label{hy11}
\end{align}
and  $(\ \dot{ }\ )$  denotes derivate with respect to the third coordinate.

Our main results in this work are the following Calabi-Bernstein type classification results:
\setcounter{theorem}{0}
\renewcommand{\thetheorem}{\Alph{theorem}}
\begin{theorem}
\label{Main1} Consider $\varphi:\R \rightarrow \mathbb{R}$  satisfying \eqref{hy1}  and 
let  $\Sigma$ be a complete $[\varphi,\vec{e}_{3}]$-minimal surface in $\mathbb{R}^{3}$ with  bounded mean curvature $H$ and  Gauss curvature $K\geq 0$. If $\dot{\varphi}$ has the following asymptotic behaviour
\begin{align}
& \langle p, \vec{e}_3\rangle \, \vert\dot{\varphi} \vert\leq C \vert p\vert, \quad p\rightarrow \infty,
\end{align}
and the Gauss map $N:\Sigma\rightarrow\mathbb{S}^{2}$ of $\Sigma$ satisfies 
\begin{equation}
\langle N,y_{0}\rangle\leq -2\varepsilon, \qquad \text{on $\Sigma$},\label{N1}
\end{equation}
for some constants $C>0$, $\varepsilon>0$, where 
$\vert \cdot \vert:\Sigma\rightarrow \mathbb{R}$ is the extrinsic distance with respect to the origin   and $y_{0}\in\mathbb{S}^{2}$,  $\langle y_0,\vec{e}_{3}\rangle =0$, then $\Sigma$ must be a plane.
\end{theorem}
\begin{theorem}
\label{Main2} Consider $\varphi:\R \rightarrow \mathbb{R}$  satisfying \eqref{hy1} and let
 $\Sigma$ be a properly embedded  $[\varphi,\vec{e}_{3}]$-minimal surface in $\mathbb{R}^{3}$ with bounded mean curvature $H$ and   Gauss curvature $K\geq 0$. If $\dot{\varphi}$ has the following asymptotic behaviour
\begin{align}
&  \vert\dot{\varphi} \vert\leq C \vert p\vert \log\vert p\vert, \quad p\rightarrow \infty,\label{hy3}
\end{align}
and the Gauss map $N$ of $\Sigma$ satisfies 
\eqref{N1}, then $\Sigma$ must be a plane.
\end{theorem}
\begin{theorem}\label{Main3} Consider $\varphi:]a,b[ \rightarrow \mathbb{R}$  satisfying \eqref{hy11} and 
\begin{align}
&  \vert\dot{\varphi} \vert\leq C \vert p\vert, \quad p\rightarrow \infty.\label{hy4}
\end{align}
Then, there is no properly embedded $[\varphi,\vec{e}_{3}]$-minimal surfaces $\Sigma$ in $\mathbb{R}^{2}\times ]a,b[$ whose Gauss map $N$ verifies 
\begin{equation}
\langle N,\vec{e}_3\rangle \geq \epsilon>0 \label{N2}
\end{equation} 
and such that either $\inf_\Sigma \dot{\varphi}^2>0$ or $\inf_\Sigma \ddot{\varphi}>0$.
\end{theorem}
Let us make some comments about these theorems. Condition  \eqref{N1} means that the image of the Gauss map lies onto a compact region in the open hemisphere bounded by the semicircle determined by a vertical plane 
through the origen. In particular, from the Ekeland's principle \cite[Proposition 2.2]{AMR},   the height function $\mu(p)= \langle p, \vec{e}_{3}\rangle$ satisfies that
$$\mu_\star = \inf_\Sigma \mu = -\infty, \qquad   \mu^\star = \sup_\Sigma \mu = +\infty,$$ otherwise, there exists a sequence of points $\{q_n\}$ in $\Sigma$ such that $N(q_n)\rightarrow \vec{e}_3$. And so,  under the assumption \eqref{N1}, the function $\varphi$ must be globally defined on $\R$. 

We should point out too, see \cite{HIMW,MM},  that tilted grim reapers are examples of flat  $[\varphi,\vec{e}_{3}]$-minimal surfaces in $\mathbb{R}^{3}$ whose Gauss map takes values  onto a non-compact semicircle on the open hemisphere bounded by a vertical plane through the origin. Moreover,  about the condition \eqref{N2}, it is remarkable,  see \cite{MM}, that the Gauss maps  of $[\varphi,\vec{e}_{3}]$-catenary cylinders and  $[\varphi,\vec{e}_{3}]$-bowls takes their values on an open hemisphere bounded by the horizontal plane through the origin. 

The proofs of Theorems \ref{Main1}, \ref{Main2} and \ref{Main3} are inspired by ideas used in \cite{BS,K} for translating solitons and depend on a blend of ideas from surface theory, elliptic theory and maximum principles.

In the ambient  $\mathbb{L}^{3}$ we prove:
\begin{theorem}
\label{Main4} Consider $\varphi:]a,b[ \rightarrow \mathbb{R}$,  a concave function, that is, $\ddot{\varphi}\leq 0$ on $]a,b[$. Then, there is no complete spacelike $[\varphi,\vec{e}_{3}]$-maximal surfaces $\Sigma$ in $\mathbb{L}^{3}$  satisfying  that 
\begin{equation}\label{loeq}
\text{either}\quad  \inf_\Sigma \dot{\varphi}^2 >0, \quad \text{or}\quad   \sup_\Sigma \ddot{\varphi}<0.
\end{equation}
\end{theorem}
\setcounter{theorem}{0}
\renewcommand{\thetheorem}{\arabic{theorem}}
\numberwithin{theorem}{section}

The paper is organized as follow: In Section \ref{s2} we recall some facts about the weak maximum principle and the stochastic completeness. After, in Section \ref{s3} we consider the Euclidean case and prove, assuming reasonable curvature restrictions,  some Calabi-Bernstein type  and non-existence results concerning complete $[\varphi,\vec{e}_3]$-minimal surfaces in $\R^3$ whose Gauss maps lie on compact subsets of open hemispheres of $\mathbb{S}^2$. Finally, Section \ref{s4} deals with the Lorentzian case where we prove a general non-existence result for complete spacelike $[\varphi,\vec{e}_3]$-maximal surfaces and, in particular, we give a Calabi-Bernstein type result when $\dot{\varphi}$ is bounded.
\section{The weak maximum principle}\label{s2}
Stochastic completeness is the property for a stochastic process to have infinite (intrinsic) life time. An analytic condition to express stochastic completeness, see  \cite[Section 2.3]{AMR},  is the following, 
\begin{definition}{\rm 
 A Riemannian manifold $(\Sigma,\langle\cdot,\cdot\rangle)$ is said to be \textit{stochastically complete} if the \textit{weak maximum principle} hold for the Laplacian operator $\Delta$, that is, if for any function $u\in C^{2}(\Sigma)$ with $u^{*}=\sup_{\Sigma}u<+\infty$, there exists a sequence of points $\{p_{n}\}_{n\in\mathbb{N}}\subset\Sigma$ satisfying
$$u(p_{n})>u^{*}-\frac{1}{n}, \ \ \text{and} \ \ \Delta u(p_{n})<\frac{1}{n}, \text{ for any } n\in\mathbb{N}.$$
}
\end{definition}
Notice that, in the above definition the Riemannian manifold $\Sigma$ is not assumed to be complete (or geodesically complete). From \cite[Theorem 2.8]{AMR}, we have that  the following statements are equivalent. 
\begin{itemize}
\item $\Sigma$ is stochastically complete.
\item For every $\lambda>0$, the only nonnegative bounded $C^{2}$ solution of $\Delta u\geq\lambda u$ on $\Sigma$ is the constant zero.
\item For every $\lambda>0$, the only nonnegative bounded $C^{2}$ solution of $\Delta u=\lambda u$ on $\Sigma$ is the constant zero.
\end{itemize}
In this paper we  shall apply the weak maximum principle for the drift laplacian operator $$\Delta^{f}(\cdot)=\Delta(\cdot)+\langle\nabla f,\nabla (\cdot)\rangle,$$
for some $f\in C^{2}(\Sigma)$, where $\nabla$ denote the gradient operator in $\Sigma$. In this sense, we may give the following definition,
\begin{definition}{\rm 
Let $(\Sigma,\langle\cdot,\cdot\rangle)$ be a Riemannian manifold and consider $f\in C^{2}(\Sigma)$. We say that $\Sigma$ is $f$-\textit{stochastically complete} if the weak maximum principle holds for $\Delta^{f}$, that is,  if for any function $u\in C^{2}(\Sigma)$ with $u^{*}=\sup_{\Sigma}u<+\infty$, there exists a sequence of points $\{p_{n}\}_{n\in\mathbb{N}}\subset\Sigma$ satisfying
$$u(p_{n})>u^{*}-\frac{1}{n}, \ \ \text{and} \ \ \Delta^{f} u(p_{n})<\frac{1}{n}, \text{ for any } n\in\mathbb{N}.$$}
\end{definition}
As above, see \cite[Theorem 2.14]{AMR}, the following characterization holds.
\begin{proposition}
The following statements are equivalent
\begin{itemize}
\item $\Sigma$ is $f$-stochastically complete.
\item For every $\lambda>0$, the only nonnegative bounded $C^{2}$ solution of $\Delta^{f} u\geq\lambda u$ on $\Sigma$ is the constant zero.
\item For every $\lambda>0$, the only nonnegative bounded $C^{2}$ solution of $\Delta^{f} u=\lambda u$ on $\Sigma$ is the constant zero.
\end{itemize}
\end{proposition}
Notice that, it is difficult to check when a Riemannian manifold verifies the $f$-weak maximum principle. To end this part we give a criteria to verifying this property when $u^{*}>0$. From Theorem \cite[Theorem 2.14]{AMR}, we can adapt the arguments of the \cite[Theorem 2.9]{AMR} to prove the following result,
\begin{theorem}
\label{criteria}
Let  $\Sigma$ be a Riemannian manifold and $f\in C^{2}(\Sigma)$. If there exists a function $\gamma\in C^{\infty}(\Sigma)$ such that $\gamma(p)\rightarrow +\infty$ as $p\rightarrow \infty$ and $\Delta^{f}\gamma\leq\, \lambda\gamma$ outside a compact set, for some $\lambda>0$, then the weak maximum principle holds in $\Sigma$ for any function $u\in C^{2}(\Sigma)$ with $0<u^{*}<+\infty$.
\end{theorem}
In particular, if the Omori-Yau maximum principle holds in $\Sigma$ for $\Delta^{f}$, see  \cite[Theorem 3.2]{AMR}, then $\Sigma$ is $f$-stochastically complete.
\section{The Euclidean case}\label{s3}
Let $\Sigma$ be the $[\varphi,\vec{e}_{3}]$-minimal immersion in $\mathbb{R}^{3}$  whose Gauss  $N:\Sigma\rightarrow\mathbb{S}^{2}$ verifies \eqref{N1}. We shall denote by  $\mathcal{S}$  the scalar second fundamental form given by $\mathcal{S}(u,v)=-\langle \textbf{S} u,v\rangle$ for any vector fields $u,v\in T\Sigma$, where $\textbf{S}$ is shape operator. 
\

Consider the height function $\mu:\Sigma\rightarrow\mathbb{R}$,  $\mu(p)=\langle p,\vec{e}_{3}\rangle$ and the angle function $\eta:\Sigma\rightarrow\mathbb{R}$,  $\eta(p)=\langle N(p),\vec{e}_{3}\rangle$. Then, from the Simons-type formula, we have, see \cite[Theorem 3]{CMZ}),
\begin{equation}
\label{b3}
\Delta^{\varphi}\vert\mathcal{S}\vert^{2}=2\vert\nabla\mathcal{S}\vert^{2}-2\vert\mathcal{S}\vert^{4}+2\ddot{\varphi}\eta^{2}\vert\mathcal{S}\vert^{2}-4\ddot{\varphi}\vert\textbf{S}\nabla\mu\vert^{2}-2\dddot{\varphi}\eta\langle\textbf{S}\nabla\mu,\nabla\mu\rangle.
\end{equation}
From \eqref{N1}, there exists  a constant $b>1$ such that the  function  $\phi:\Sigma\rightarrow\mathbb{R}$ given  by $\phi(p)=1-\langle N(p),y_{0}\rangle$ satisfies that $\phi-b\geq\varepsilon$. Moreover, by a straightforward computation, we have that $\nabla\phi=\textbf{S}y_{0}^{\top}$, where $(\ .\ )^{\top}$ denotes the projection on the tangent bundle, and
\begin{align} 
\label{c1}
\Delta^{\varphi}\phi=(\vert\mathcal{S}\vert^{2}-\ddot{\varphi}\eta^{2})(1-\phi).
\end{align}
Let $\psi:\Sigma\rightarrow\mathbb{R}$ be the function given by 
$$\psi=\frac{\vert\mathcal{S}\vert^{2}}{(b-\phi)^{2}}.$$
From equations \eqref{b3} and \eqref{c1}, we obtain that
\begin{align}
\label{c2}
\Delta^{\varphi}\psi&=\frac{2\vert\nabla\mathcal{S}\vert^{2}}{(b-\phi)^{2}}-\frac{2\vert\mathcal{S}\vert^{4}}{(b-\phi)^{2}}+\frac{2\ddot{\varphi}\eta^2\vert\mathcal{S}\vert^{2}}{(b-\phi)^{2}}-\frac{4\ddot{\varphi}\vert\textbf{S}\nabla\mu\vert^{2}}{(b-\phi)^2}-\frac{2\dddot{\varphi}\eta\langle\textbf{S}\nabla\mu,\nabla\mu\rangle}{(b-\phi)^2} \\
&+\frac{2\vert\mathcal{S}\vert^{4}(1-\phi)}{(b-\phi)^{3}}-\frac{2\ddot{\varphi}\eta^2\vert\mathcal{S}\vert^2(1-\phi)}{(b-\phi)^3}+\frac{6\vert\mathcal{S}\vert^{2}\vert\nabla\phi\vert^{2}}{(b-\phi)^4}+\frac{4\langle\nabla\vert\mathcal{S}\vert^{2},\nabla\phi\rangle}{(b-\phi)^3}\nonumber.
\end{align}
But,
\begin{align}
&\frac{\langle\nabla\phi,\nabla\psi\rangle}{b-\phi}=\frac{\langle\nabla\phi,\nabla\vert\mathcal{S}\vert^{2}\rangle}{(b-\phi)^{3}}+\frac{2\vert\mathcal{S}\vert^{2}\vert\nabla\phi\vert^{2}}{(b-\phi)^{4}}, \label{c3}\\
&\frac{2\vert\nabla \mathcal{S}\vert^{2}}{(b-\phi)^{2}}+\frac{2\vert\nabla\phi\vert^{2}\vert\mathcal{S}\vert^{2}}{(b-\phi)^{4}}\geq\frac{4\vert\nabla\phi\vert \vert\mathcal{S}\vert \vert \nabla \mathcal{S}\vert}{(b-\phi)^{3}}. \label{c4} ,
\end{align}
and so, from \eqref{c2},\eqref{c3} and \eqref{c4}, we have the following inequality
\begin{align}
\label{c5}
\Delta^{\varphi}\psi\geq& \frac{2(1-\phi)\vert\mathcal{S}\vert^{4}}{(b-\phi)^{3}}-\frac{2\vert\mathcal{S}\vert^{4}}{(b-\phi)^{2}}+\frac{2\langle\nabla\phi,\nabla\psi\rangle}{b-\phi}+ \frac{2\ddot{\varphi}\eta^2\vert\mathcal{S}\vert^{2}}{(b-\phi)^2}\\
&-\frac{2\ddot{\varphi}\vert\mathcal{S}\vert^{2}\eta^{2}(1-\phi)}{(b-\phi)^{3}}-\frac{4\ddot{\varphi}\vert\textbf{S}\nabla\mu\vert^{2}}{(b-\phi)^{2}}-\frac{2\dddot{\varphi}\eta\langle\textbf{S}\nabla\mu,\nabla\mu\rangle}{(b-\phi)^{2}}.\nonumber
\end{align}
\textbf{Claim 1:} The following statements hold on $\Sigma$
\begin{enumerate}
\item [(i)] $\ddot{\varphi}\eta^2\vert\mathcal{S}\vert^{2}(b-\phi)-\ddot{\varphi}\vert\mathcal{S}\vert^{2}\eta^{2}(1-\phi)-2\ddot{\varphi}\vert\mathcal{S}\nabla\mu\vert^{2}(b-\phi)\leq 0,$
\item [(ii)] $\dddot{\varphi}\eta\langle\textbf{S}\nabla\mu,\nabla\mu\rangle\leq 0$.
\end{enumerate}
\begin{proof}[Proof of  Claim 1.]

Fix any point $p\in\Sigma$. The first item (i) trivially holds since $-\ddot{\varphi}\vert\textbf{S}\nabla\mu\vert^{2}\geq 0$ and $b-\phi\geq 1-\phi$. On the other hand, by taking  an orthornormal frame of principal directions $\{v_{i}\}_{i=1,2}$ of $T_{p}\Sigma$,  $$\langle\textbf{S}\nabla\mu,\nabla\mu\rangle=\sum_{i=1}^{2}\langle\nabla\mu,v_{i}\rangle^{2}k_{i},$$
where $k_{i}$ are the principal curvatures of $\Sigma$ in $p$. 

If $\dot{\varphi}\geq 0$ everywhere (the case $\dot{\varphi}\leq 0$ is analogous), then as $K\geq 0$, we get that either $H\leq 0$ or $H\geq 0$ on $\Sigma$.
\begin{itemize}
\item If $H\leq 0$, then each $k_{i}\leq 0$ on $\Sigma$. In particular, from the equation \eqref{meancurvature}, the angle function $\eta\geq 0$. Consequently, from the condition $(\dagger)$, the following inequality holds
\begin{equation}
\label{ineqq}
0\geq \dddot{\varphi}\eta\langle\textbf{S}\nabla\mu,\nabla\mu\rangle=\sum_{i=1}^{2}\dddot{\varphi}\eta\, k_{i}\langle\nabla\mu,v_{i}\rangle^{2}.
\end{equation}
\item If $H\geq 0$, then each $k_{i}\geq 0$ on $\Sigma$. In this case, the angle function $\eta\leq 0$ and then, the inequality \eqref{ineqq} also holds.
\end{itemize}
\end{proof}

Consequently, from the inequality \eqref{c5} together with Claim 1, if we denote by ${\cal J}$ to the drift Laplacian operator   $\Delta^{\varphi+2\text{log}(\phi-b)}$, then
\begin{equation}
\label{c6}
{\cal J} \psi \geq 2(b- \phi)(1-b)\, \psi^{2}\geq 2\varepsilon (b-1)\psi^{2}.
\end{equation}
Let  $\rho(p)=|p|$  be the distance function in $\R^3$ from the point $p\in \Sigma$ to the origin and $\lambda = (R^2-\rho^2)^2 \psi$ where $R$ is any positive real number. It is easy to check that 
\begin{align}
 \rho \nabla \rho &= p^\top,\label{grho}\\
{\cal J} \rho^2 &= 2(2 +\dot{\varphi} \mu - \frac{2}{b-\phi} {\cal S}(y_0^\top,p^\top),\label{lrho}\\
{\cal J}  \lambda&=(R^2-\rho^2)^2 {\cal J}\psi - 2 \psi (R^2-\rho^2) {\cal J}\rho^2  \label{llambda} \\ 
&+ 8 |p^\top|^2 \psi - 8 (R^2-\rho^2) \langle \nabla \psi , p^\top\rangle.\nonumber
\end{align}

\subsection{ Proof of Theorem \ref{Main1}}
By using that $H$ is bounded on $\Sigma$ and $K\geq 0$,  we have that the length of the second fundamental form $|{\cal S}|^2 = H^2 - K \leq H^2 $  is uniformly bounded on $\Sigma$ and, from \eqref{c6}, \eqref{lrho} and \eqref{llambda}, there exist positive constants $C_1$ and $C_2$ such that
\begin{align}
{\cal J} \lambda &\geq 2(R^2-\rho^2)^2 \varepsilon \, (b-1)\psi^2 - C_1\psi (R^2-\rho^2) (C_2 + \rho)\label{dlambda}\\
&+ 8 |p^\top|^2 \psi - 8 (R^2-\rho^2) \langle \nabla \psi, p^\top\rangle.\nonumber
\end{align}
Let $B_R$ be the ball in $\R^3$ of radius $R$ with center at the origin and consider $R$ large enough so that $\Sigma_R= \Sigma\cap B_R\neq \emptyset$. Then, on $\Sigma_R$ the function $\lambda$ attains its maximum at a interior point $p_R$, where we have that
$$ \nabla \lambda (p_R) =0, \qquad {\cal J} \lambda (p_R) \geq 0.$$ 
Thus, from the above expression, \eqref{grho}, \eqref{dlambda} and by a straightforward computation, we get that,
\begin{equation}\sup_{\Sigma_R} \lambda\leq \widetilde{C}_1R^2 (\widetilde{C}_2 + R),
\end{equation}
for some positive constants $\widetilde{C}_1$ and $\widetilde{C}_2$.
In particular on $\Sigma_{R/2}$ the function $\psi$ satisfies
$$ \psi \leq \frac{16}{9}\widetilde{C}_1 \frac{\widetilde{C}_2 + R}{R^2},$$
that is, $\psi\rightarrow 0$ as $R\rightarrow \infty$ and then, $\Sigma$ must be a plane.
\begin{proof}[]
\end{proof}
\subsection{Proof of Theorem \ref{Main2}}
Let $\Sigma$ be a properly embedded $[\varphi,\vec{e}_{3}]$-minimal surface in $\mathbb{R}^{3}$. Up to horizontal translations, we can assume $0$ is not contained in $\Sigma$. Consider the following smooth function $\gamma:\Sigma\rightarrow\mathbb{R}$ given by $\gamma(p)=2\log(\vert p\vert)$. As $\Sigma$ is properly embedded, we have that $\vert\gamma\vert\rightarrow+\infty$ when $p\rightarrow+\infty$.  It is not difficult to check that
\begin{align*}
&\Delta\gamma=-4\frac{\vert p^{T}\vert}{\vert p\vert^{2}}+\frac{2}{\vert p\vert^{2}}\left(2+\langle p,N\rangle H \right), \\
&\langle\nabla\gamma,\nabla(\varphi+2\log(\phi-b))\rangle=\frac{1}{\vert p\vert^{2}}\left(2\dot{\varphi}\langle p^{T},\vec{e}_{3}^{T}\rangle+4\frac{\mathcal{S}(p^{T},y_{0}^{T})}{\phi-b}\right). 
\end{align*}
Now, from the above expressions, \eqref{hy3} and  by using that $H$ is bounded and $K\geq 0$, we have  that  $\vert\mathcal{S}\vert^{2}$ and $\psi$ are  uniformly bounded and  there must be a constant $C>0$ such that  $${\cal J} \gamma=\Delta^{\varphi+2\log(\phi-b)}\gamma\leq C\gamma,$$ outside a compact set. Thus, from the Proposition \ref{criteria}, $\Sigma$ is $(\varphi+2\log(\phi-b))$-stochastics complete  and there exists a sequence of points $\{p_{n}\}_{n\in\mathbb{N}}\subset\Sigma$ satisfying
$$\psi(p_{n})\geq\psi^{*}-\frac{1}{n} \ \ \text{and} \ \ {\cal J}\psi(p_{n})\leq \frac{1}{n} \ \ \text{ for any }n\in\mathbb{N}, $$
where $\psi^{*} = \sup_\Sigma \psi$.   However, by taking limits in \eqref{c6}, we obtain that $$0\geq 2\varepsilon (b-1)(\psi^{*})^{2}\geq 0,$$ 
which proves that $\psi^{*}=0$ and $\Sigma$ must be a plane.
\begin{proof}[]\end{proof}
\subsection{Proof of Theorem \ref{Main3}}
Let $\eta =\langle N,\vec{e}_3\rangle$ be the angle function of $\Sigma$,  then by using  \cite[Lemma 2.1]{MM}, we get
\begin{equation}
\Delta^\varphi \eta = - (\ddot{\varphi} \vert \nabla \mu \vert^2 + \vert A\vert^2)\eta \leq -(\ddot{\varphi} \vert \nabla \mu \vert^2 + \frac{\dot{\varphi}^2 \eta^2}{2})\eta.\label{thc}
\end{equation}
Consider the function $\gamma:\Sigma\longrightarrow \R$ given by $\gamma(p) = 2 \log \vert p\vert$, then as $\Sigma$ is properly embedded and $\varphi$ satisfies \eqref{hy4},  we can check that
\begin{align}
&\gamma(p) \rightarrow +\infty, \quad \text{  $p\rightarrow \infty$} \label{comp1}\\
&\vert \nabla \gamma(p)\vert = 2\frac{\vert p^\top\vert^2}{\vert p\vert^2} \leq 2,\quad p \rightarrow \infty\label{comp2}\\
& \Delta^\varphi \gamma(p) = -4\frac{\vert p^\top\vert^2}{\vert p\vert^4} + \frac{2 \mu(p) \dot{\varphi}(\mu(p))+4}{\vert p\vert^2} \leq C, \quad p\rightarrow \infty.\label{comp3}
\end{align}
Thus,  from \cite[Theorem 3.2]{AMR} we can apply the generalized Omori-Yau maximum principle to $\Delta^\varphi$ and there exists a sequence of points $\{p_n\}$ in $\Sigma$ such that 
\begin{align*}
&\eta(p_n) \rightarrow \eta_\star =\inf_\Sigma\eta >0\\
&\Delta^\varphi\eta(p_n) > \frac{1}{n}, \quad \nabla\eta (p_n) \rightarrow 0,
\end{align*}
and taking limits in \eqref{thc} along the points $p_n$, we get a contradiction.
\begin{proof}[]\end{proof}
\section{The Lorentzian case}\label{s4}
Let $\mathbb{L}^3$ be  the 3-dimensional Lorentz-Minkowski space, that is, the real vector space $\mathbb{R}^3$ endowed with the Lorentzian metric tensor 
$$\langle \cdot, \cdot\rangle_{\mathbb{L}^3} = dx^2+ dy^2- dz^2$$
where $(x,y,z)$ are the canonical coordinates of $\R^3$. A surface $\Sigma$ in $\L^3$  is said to be a spacelike surface if the induced metric  is a Riemannian metric on $\Sigma$, which, as usual, is also denoted by $\langle \cdot, \cdot\rangle_{\mathbb{L}^3} $. It is well-known that such a surface is orientable, namely, we can choose
a unit timelike normal vector field $N$ globally defined on $\Sigma$  in the same time-orientation of $\vec{e}_3$ that we will call the Gauss map of the immersion.
\\
Denote by $\overline{D}$ and $D$, the Levi-Civita connections of $\mathbb{L}^{3}$ and $\Sigma$, respectively. Then the Gauss and Weingarten formulas for $\Sigma$  in $\L^3$ are given by
\begin{align}
& \overline{D}_XY = D_X Y + {\cal S}(X,Y)N =D_X Y - \langle A(X), Y \rangle N \\
&A(X) = - dN(X) = -  \overline{D}_XN
\end{align} 
 for any $X,Y\in T\Sigma$. We will consider $H=\text{trace}(\mathcal{S})$ and $K = -\det(A)$ the mean and Gaussian curvatures of $\Sigma$. 

\

Let $\varphi:]a,b[\rightarrow\mathbb{R}$ be a smooth function. Recall that $\Sigma$ is $[\varphi,\vec{e}_{3}]$-\textit{maximal} in $\mathfrak{D}^3=\R^2\times]a,b[\subseteq\mathbb{L}^{3}$ if and only if the mean curvature of $\Sigma$ satisfies 
\begin{equation}
\label{a1L3}
H= \dot{\varphi} \  \langle N,\vec{e}_3\rangle_{\L^3},
\end{equation}
which is equivalent to say that $\Sigma$ is a critical point of the weighted area functional \eqref{area}. 

Consider the height function $\mu=-\langle\psi,\vec{e}_{3}\rangle_{\mathbb{L}^{3}}$  and the hyperbolic angle function $\eta=\langle N,\vec{e}_{3}\rangle_{\mathbb{L}^{3}}$  of a spacelike $[\varphi,\vec{e}_{3}]$-maximal surface in $\L^3$,  then arguing as in \cite[Lemma 2.1]{MM} we have,
\begin{lemma}
\label{equations} The following statements hold
\begin{enumerate}
\item $\nabla\mu=-\vec{e}_{3}^{\top}$ , $\vert\nabla\mu\vert_{\L^3}^{2}=\eta^{2}-1$,
\item $\langle\nabla\eta,\cdot\rangle_{\L^3}=-\mathcal{S}(\nabla\mu,\cdot)$,
\item $\nabla^{2}\mu(\cdot,\cdot)=-\eta\mathcal{S}(\cdot,\cdot)$,
\item $\nabla^{2}\eta(\cdot,\cdot)=-(D_{\nabla\mu}\mathcal{S})(\cdot,\cdot)+\eta\mathcal{S}^{[2]}(\cdot,\cdot),$
\item $\Delta^{\varphi}\eta=\eta\left(\vert\mathcal{S}\vert_{\L^3}^{2}-\ddot{\varphi}\vert\nabla\mu\vert_{\L^3}^{2} \right)$,
\end{enumerate}
where $^\top$, $\nabla$ and $\nabla^2$ denote the projection on the tangent bundle, the gradient and hessian operators, respectively,   in $\Sigma$ and $\mathcal{S}^{[2]}$ is the  symmetric $2$-tensor given by
$$\mathcal{S}^{[2]}(X,Y)=\sum_{i=1}^{2}\mathcal{S}(X,v_{i})\mathcal{S}(Y,v_{i})$$
with $\{v_{i}\}_{i=1,2}$ an orthonormal frame of $T\Sigma$.
\end{lemma}
\begin{lemma}\label{ll2}
Let $\Sigma$ be a spacelike $[\varphi,\vec{e}_{3}]$-maximal surface in $\mathbb{L}^{3}$ and consider the smooth function $\Upsilon:\Sigma\rightarrow\mathbb{R}$ given by $\Upsilon=-(1+\eta^{2})^{-1/2}$. Then, the following inequality hold
\begin{equation}-\Upsilon\Delta^{\varphi}\Upsilon+3\vert\nabla\Upsilon\vert^{2}\geq (\vert\mathcal{S}\vert^{2}-\ddot{\varphi}\vert\nabla\mu\vert^{2})\frac{1-\Upsilon^{2}}{1+\eta^{2}}\geq 0.\label{oylorentz}\end{equation}
\end{lemma} 
\begin{proof}
By a straightforward computation,
$$\Delta^{\varphi}\Upsilon=\frac{\Delta^{\varphi}\eta^{2}}{2(1+\eta^{2})^{3/2}}-\frac{3\vert\nabla\eta^{2}\vert^{2}}{4(1+\eta^{2})^{5/2}}.$$
Moreover, from item 5. of Lemma \ref{equations} and multiplying in both members by $(1+\eta^{2})^{-1/2}$, we obtain that
$$\frac{\Delta^{\varphi}\Upsilon}{(1+\eta^{2})^{1/2}}\geq (\vert\mathcal{S}\vert^{2} -\ddot{\varphi}\vert\nabla\mu\vert^{2})\frac{\eta^{2}}{(1+\eta^{2})^{2}}-\frac{3\vert\nabla\eta^{2}\vert^{2}}{4(1+\eta^{2})^{3}}.$$
Finally, the proof follows taking into account that
$$1-\Upsilon^{2}=\frac{\eta^{2}}{1+\eta^{2}}  \ \ \text{and} \ \ \nabla\Upsilon=\frac{\nabla\eta^{2}}{2(1+\eta^{2})^{3/2}}$$
\end{proof}
\begin{lemma} \label{ll3}
Let $\Sigma$ be a complete spacelike $[\varphi,\vec{e}_{3}]$-maximal surface in $\mathbb{L}^{3}$. If $\ddot{\varphi} \leq0$ on $]a,b[$, then the generalized Omori-Yau maximum principle can be applied on $\Delta^\varphi$.
\end{lemma}
\begin{proof}
By taking the Bakry-\'Emery Ricci tensor which is defined by 
$$ Ric_\varphi = Ric - \nabla^2 \varphi,$$
where $Ric$ denotes the standard Ricci tensor on $\Sigma$, we have from the Gauss equation, Lemma \ref{ll2} and by a straightforward computation  that 
\begin{align*}
& Ric_\varphi(X,X) = Ric(X,X) - \nabla^2 \varphi(X,X) = \vert AX\vert^2_{\L^3} - \ddot{\varphi}\langle X,\nabla \mu\rangle_{\L^3}^2 \geq 0.
\end{align*}
Then, from completeness and by applying the weighted  Bochner's formula to the distance function $r(\, \cdot \, )= d_\Sigma (\cdot,p_0)$, see \cite[Theorem 1.1]{WW}, we have
$$ \Delta^\varphi r (p) \leq \frac{C}{r(p)} \leq C, \quad p\rightarrow \infty.$$
This expression together with the fact that $\vert \nabla r\vert=1$ allows us to apply Theorem 3.2 in \cite{AMR} and to  conclude the proof.
\end{proof}
\subsection{Proof of Theorem \ref{Main4}}
We may assume that $\eta$ is non constant on $\Sigma$ otherwise $\Sigma$ must be an horizontal plane and $\dot{\varphi}\equiv 0$. In this case, from Lemma \ref{ll3}, there exists a sequence $\{p_n\} \subset \Sigma$ such that,
\begin{align}
& \Upsilon(p_n)\rightarrow \sup_\Sigma\Upsilon =\Upsilon^*, \quad -\frac{1}{\sqrt{2}}< \Upsilon^* \leq 0,\label{le1}\\
&\Delta^{\varphi} \Upsilon(p_{n})<\frac{1}{n}, \text{ for any } n\in\mathbb{N}, \quad \vert \nabla \Upsilon\vert(p_n)\rightarrow 0.\label{le2}
\end{align}
Now, having in mind that $2\vert {\cal S}\vert^2 \geq H^2 = \dot{\varphi}^2 \eta^2$ and 
$$ \frac{\vert \nabla \mu \vert^2}{\eta^2} (p_n) \rightarrow \frac{-1}{\eta^2_*} + 1, \qquad \eta^2_* =\sup_\Sigma \eta^2,$$
if we plug the sequence in the inequality \eqref{oylorentz} and take the limit, then  
\begin{equation} \inf_\Sigma \dot{\varphi}^2 = 0\ \ \text{and}\ \  \sup_\Sigma \ddot{\varphi} = 0, \label{condlorentz}
\end{equation}
which concludes the proof.
\begin{proof}[]\end{proof}
\begin{corollary}
Let $\varphi:]a,b[\rightarrow \R$ be an analytic function with $\dot{\varphi}$  bounded  and satisying \eqref{hy1}. Then any complete spacelike $[\varphi,\vec{e}_3]$-maximal surface in $\L^3$ must be a plane.
\end{corollary}
\begin{proof}
From \eqref{hy1} we may assume that  $\dot{\varphi}\geq 0$, $\ddot{\varphi}\leq 0$ and $\dddot{\varphi}\geq 0$ on $]a,b[$, otherwise we could argue in a similar way. Thus, from Theorem \ref{Main4}, $\dot{\varphi}$  will be a nonnegative decreasing convex function satisfying \eqref{condlorentz}. 

We assert that either $\dot{\varphi}\equiv 0$ or $]a,b[ = \R$. 
In fact, let consider $\{p_n\}\subset \Sigma$ the sequence of points satisfying \eqref{le1} and \eqref{le2}, then 
\begin{align}
&\inf_n\vert \nabla \mu\vert^2(p_n) >0, \quad \dot{\varphi}(p_n) \rightarrow 0,\label{le3}\\
& \mu_0= \sup_n\mu(p_n) \leq \mu^* = \sup_\Sigma \mu \leq b.\label{le4}
\end{align}
If $\mu_0<\mu^*$, then $\dot{\varphi}\equiv 0$ on $]\mu_0,\mu^*[$ and, by analyticity,  $\Sigma$ must be a plane.  If not,  $\mu_0=\mu^*$ and then we can apply the Ekeland's principle \cite[Proposition 2.2]{AMR} and use  \eqref{le3} and \eqref{le4} to prove that $\mu^*=\sup_\Sigma\mu=+\infty$.

Now, consider  $\mu_* = \inf_\Sigma\mu$. If $\mu_*>-\infty$, from Lemma \ref{ll3}, there exists a sequence of points $\{q_n\}$ in $\Sigma$ satisfying
$$ \mu(q_n)\rightarrow \mu_*, \quad \inf \vert\nabla u\vert(q_n) = 0, \quad \Delta^\varphi \mu (q_n)> \frac{1}{n},  $$
but, as $\dot{\varphi}$ is decreasing in $]\mu_*,\infty[$,  we have that  $\inf_n\dot{\varphi}(q_n)>0$ (otherwise $\dot{\varphi}\equiv 0$ on $]\mu_*,\infty[$ and $\Sigma$ would be a plane) and then, from Lemma \ref{equations},  the following inequalities hold
$$ 0 \leq \inf_n\dot{\varphi}(q_n)= - \inf_n \eta^2(q_n) \dot{\varphi} (q_n) < 0, $$
which is a contradiction and our assertion is true.

The results follows because a decreasing convex and bounded smooth function $\dot{\varphi}:\R\rightarrow\R$ satisfying  \eqref{condlorentz} must be identically zero and so, $\Sigma$ is a plane.

\end{proof}

\end{document}